\documentclass[11pt,a4paper,dvipsnames]{article}
\usepackage[utf8]{inputenc}

\title{Intersection Bodies of Polytopes}
\author{Katalin Berlow, Marie-Charlotte Brandenburg,\\ Chiara Meroni and Isabelle Shankar}

\date{}

\usepackage{a4wide}
\usepackage{amssymb}
\usepackage{amsmath}
\usepackage{euscript}
\usepackage{amsthm}
\usepackage{mathtools}
\usepackage{amsopn}
\usepackage{stackrel}
\usepackage[all]{xy}
\usepackage{setspace}
\usepackage{float}
\usepackage{blkarray}
\usepackage{xcolor}
\usepackage[pdftex, colorlinks, linkcolor=blue, citecolor=blue, urlcolor=blue]{hyperref}
\hypersetup{linkcolor=MidnightBlue, citecolor=MidnightBlue, urlcolor=MidnightBlue}
\usepackage{pgfplots}
\usepackage[margin=1cm]{caption}
\usepackage{subcaption}
\pgfplotsset{width=7cm, compat=1.10}
\usepgfplotslibrary{fillbetween}
\usepackage[ruled,vlined]{algorithm2e}
\usetikzlibrary{matrix}
\usepackage{mathdots}
\usepackage[nameinlink]{cleveref}
\usepackage{algorithmic}

\usepackage{todonotes}

\newcommand{\R}{\mathbb{R}}

\newcommand{\N}{\mathbb{N}}
\newcommand{\Q}{\mathbb{Q}}

\newcommand{\co}{\text{co}}
\newcommand{\conv}{\text{conv}}

\newcommand{\cube}[1]{C^{(#1)}}

\DeclareMathOperator{\vol}{Vol}
\DeclareMathOperator{\sgn}{sgn}

\newtheorem{theorem}{Theorem}[section]
\newtheorem{cor}[theorem]{Corollary}
\newtheorem{lemma}[theorem]{Lemma}
\newtheorem{prop}[theorem]{Proposition}

\theoremstyle{definition}
\newtheorem{definition}[theorem]{Definition}
\newtheorem{example}[theorem]{Example}

\theoremstyle{remark}
\newtheorem{remark}[theorem]{Remark}

\begin{document}

\maketitle
\begin{abstract}
\noindent 
We investigate the intersection body of a convex polytope using tools from combinatorics and real algebraic geometry.  In particular, we show that the intersection body of a polytope is always a semialgebraic set and provide an algorithm for its computation.  Moreover, we compute the irreducible components of the algebraic boundary and provide an upper bound for the degree of these components.
\end{abstract}

\section{Introduction}

\begin{figure}[b]
    \centering
    \includegraphics[width=0.29\textwidth]{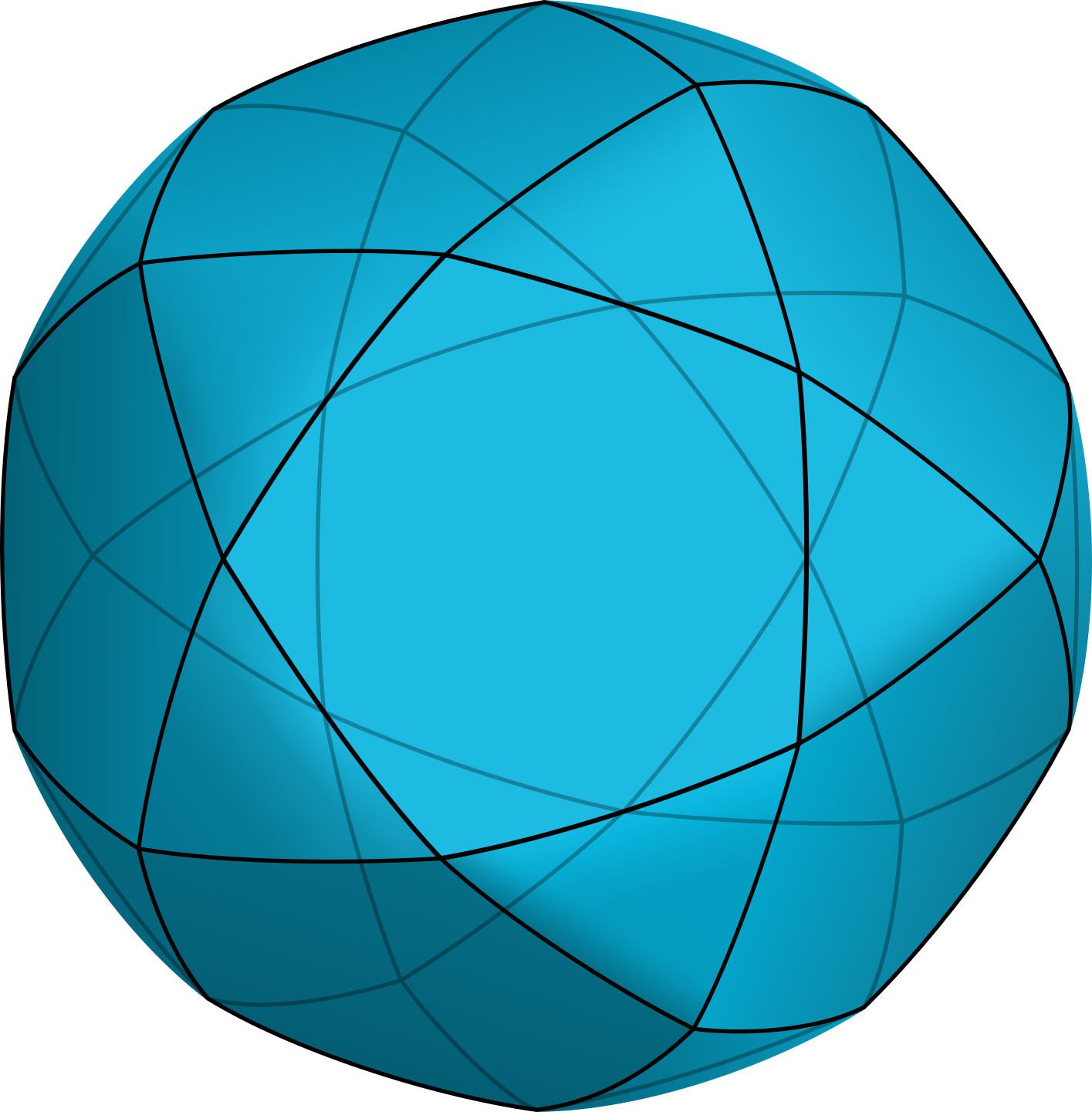}
    \caption{The intersection body of the icosahedron.}
    \label{fig:int_icosahedron}
\end{figure}
This paper studies intersection bodies from the perspective of real algebraic geometry.
Originally, intersection bodies were defined by Lutwak \cite{Lutwak1988} in the context of convex geometry. 
In view of the notion of $(d-1)$-dimensional cross-section measures  
    and the related concepts of associated
    bodies (such as intersection bodies, cross-section bodies, and projection  
    bodies), intersection bodies play an essential role in geometric  
    tomography (see \cite[Chapter 8]{G2006} and \cite[Section 2.3]{Horst1994}). In particular, we mention here the Busemann-Petty  
    problem which asks if one can compare the volumes of two convex bodies by comparing the volumes of their sections \cite{Gardner19941,Gardner19942,GKS1999,Koldobsky1998,Zhang1999}.
Moreover, Ludwig showed that the unique non-trivial GL$(d)$-covariant star-body-valued valuation on convex polytopes corresponds to taking the intersection body of the dual polytope \cite{Ludwig2006}. Due to such results, the knowledge on properties of  
    intersection bodies interestingly contributes
    also to the (still not systematized) theory of starshaped sets, see  
    Section 17 of the exposition \cite{HHMM2020}.\\

Recently, there is increased interest in investigating convex geometry from an algebraic point of view \cite{BPT2013,sinnAlgebraicBoundariesConvex2015, Sturmfels:duality, Ranestad2011}.
In this article,
we will focus on the intersection bodies of polytopes from this perspective. It is known that in $\R^2$, the intersection body of a centrally symmetric polytope centered at the origin is the same polytope rotated by $\pi/2$ and dilated by a factor of $2$ (see e.g. \cite[Theorem 8.1.4]{G2006}). Moreover, if $K$ is a full-dimensional convex body in $\R^d$ centered at the origin, then so is its intersection body \cite[Chapter 8.1]{G2006}.
But what do these objects look like in general? In $\R^d$, with $d\geq 3$, they cannot be polytopes \cite{campi_1999,zhang_1999} and they may not even be convex. In fact, for every convex body $K$, there exists a translate of $K$ such that its intersection body is not convex. This happens because of the important role played by the origin in the construction of the intersection body. \\

Our main contribution is \Cref{theorem:IP_semialgebraic}, which states that the intersection body of a polytope is a semialgebraic set, i.e. a subset of $\R^d$ defined by a boolean combination of polynomial inequalities. 
The proof relies on two key facts.  First, the volume of a polytope can be computed using determinants.  Second, the combinatorial type of the intersection of a polytope with a hyperplane is fixed for each region of a certain central hyperplane arrangement. In \Cref{section:IP_semialgebraic}, we prove semialgebraicity for the intersection body of polytopes containing the origin, and we generalize the result to arbitrary polytopes in \Cref{sec:non-convex-IP}.
In \Cref{section:algorithm}, we present an algorithm to compute the radial function of the intersection body of a polytope.  An implementation is available at \cite{mathrepo}.
In \Cref{section:algebraic_boundary}, we describe the algebraic boundary of the intersection body, which is a hypersurface consisting of several irreducible components, each corresponding to a region of the aforementioned hyperplane arrangement. \Cref{cor:degree} gives a bound on the degree of the irreducible components.
\Cref{section:cube} focuses on the intersection body of the $d$-cube centered at the origin (\Cref{fig:int_cube}).

\section{The Intersection Body of a Polytope is Semialgebraic}
\label{section:IP_semialgebraic}

In convex geometry it is common to use functions in order to describe a convex body, i.e. a non-empty convex compact subset of $\R^d$. This can be done e.g. by the radial function. A more detailed introduction can be found in \cite{schneider}.

\begin{definition}
    Given a convex body $K\subset \R^d$, the \emph{radial function} of $K$ is
    \begin{equation*}
        \rho_K : \R^d \to \R, \quad x \mapsto \max \left\{ \lambda \in \R \mid \lambda x \in K \right\}.
    \end{equation*}
\end{definition}
As a convention $\rho_K(0)$ is $\infty$ when $0\in K$ and it is $0$ otherwise. An immediate consequence of the definition is that $\rho_K(cx) = \frac{1}{c} \rho_K(x)$ for $c>0$. Therefore, we can equivalently define the radial function on the unit sphere $S^{d-1}$, and then extend to the whole space using the previously mentioned relation. Throughout this paper we will use the following convention: $x$ denotes a vector in $\R^d$ whereas $u$ denotes a vector in $S^{d-1}$. With the observation that we can restrict to the sphere, we define the intersection body of $K$ by its radial function, which is given by the volume of the intersections of $K$ with hyperplanes through the origin.

\begin{definition}
Let $K$ be a convex body in $\R^d$. 
Its \emph{intersection body} is defined to be the set $IK = \{ x\in \R^d \mid \rho_{IK}(x) \geq 1 \}$
where the radial function (restricted to the sphere) is 
$$ \rho_{IK}(u) = \text{Vol}_{d-1}(K\cap u^\perp)$$
for $u \in S^{d-1}$. We denote by $u^\perp$ the hyperplane through the origin with normal vector $u$, and by $\vol_i$ the $i$-dimensional Euclidean volume, for $i \leq d$. 
\end{definition}

We begin our investigation by considering the intersection body of polytopes which contain the origin.  
For instance, \Cref{fig:int_icosahedron} displays the intersection body of an icosahedron centered at the origin.
If the origin belongs to the interior of the polytope $P$, then $\rho_P$ is continuous and hence $\rho_{IP}$ is also continuous \cite{G2006}. Otherwise we may have some points of discontinuity which correspond to unit vectors $u$ such that $u^\perp$ contains a facet of $P$; there are finitely many such directions. The intersection body is well defined, but there may arise subtleties when dealing with the boundary. However, we will see later (in \Cref{rmk:discontinuity_radial}) that for our purposes everything works out.
In the following we use notions from polytope theory, such as \emph{zonotopes} and \emph{combinatorial types}. For further background on polytopes we refer the reader to \cite{Z1995}.

\begin{example}\label{ex:cube}
We will use the cube as an ongoing example to illustrate the key concepts used throughout the paper.
Let $P$ be the $3$-dimensional cube $[-1,1]^3 \subseteq \R^3$.
If we intersect $P$ with hyperplanes $u^\perp$, for $u\in S^2$, we can observe that there are two possible combinatorial types for $P\cap u^\perp$: it is either a parallelogram (\Cref{fig:section_parallelogram}) or a hexagon (\Cref{fig:section_hexagon}). 
\begin{figure}[t]
    \centering
    \begin{subfigure}{0.48\textwidth}
        \centering
        \includegraphics[width=0.52\textwidth]{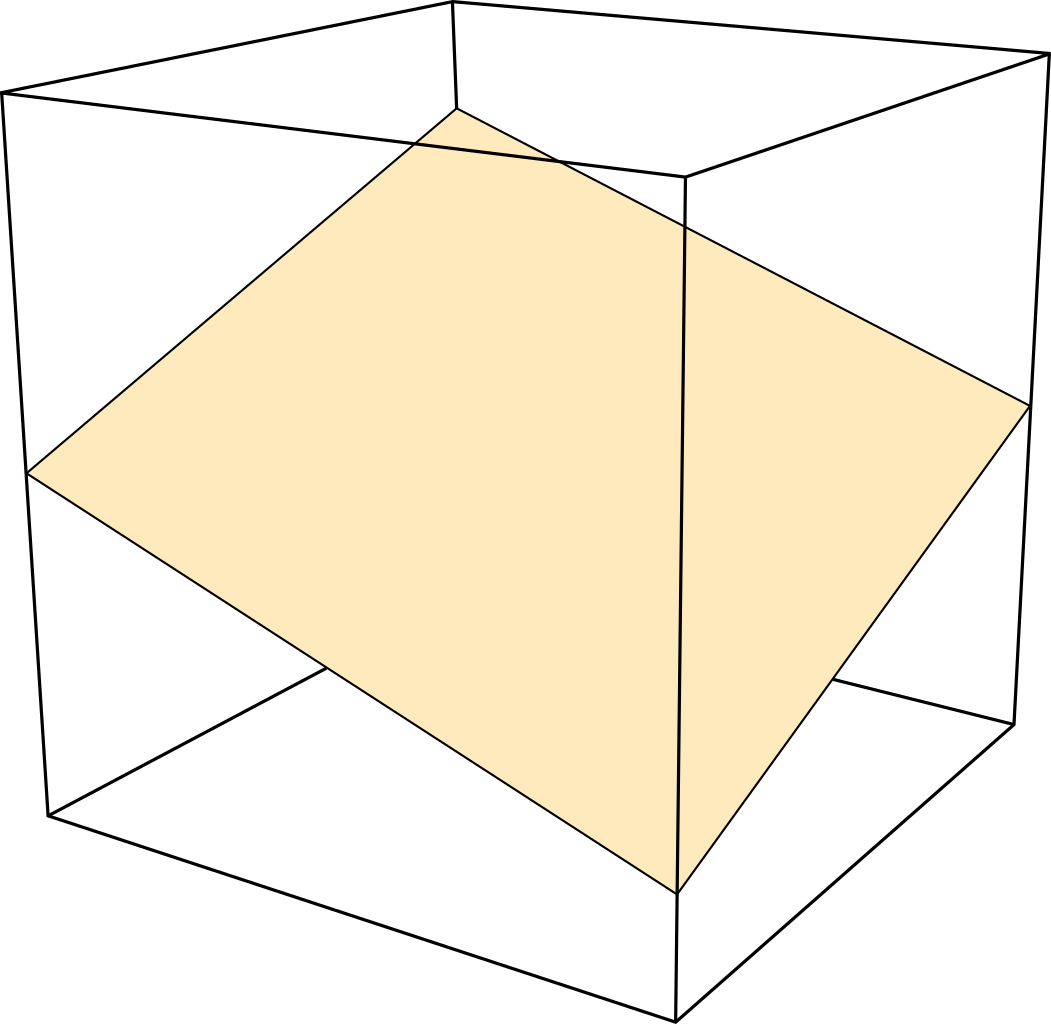}
        \caption{}
        \label{fig:section_parallelogram}
    \end{subfigure}
    \begin{subfigure}{0.48\textwidth}
        \centering
        \includegraphics[width=0.52\textwidth]{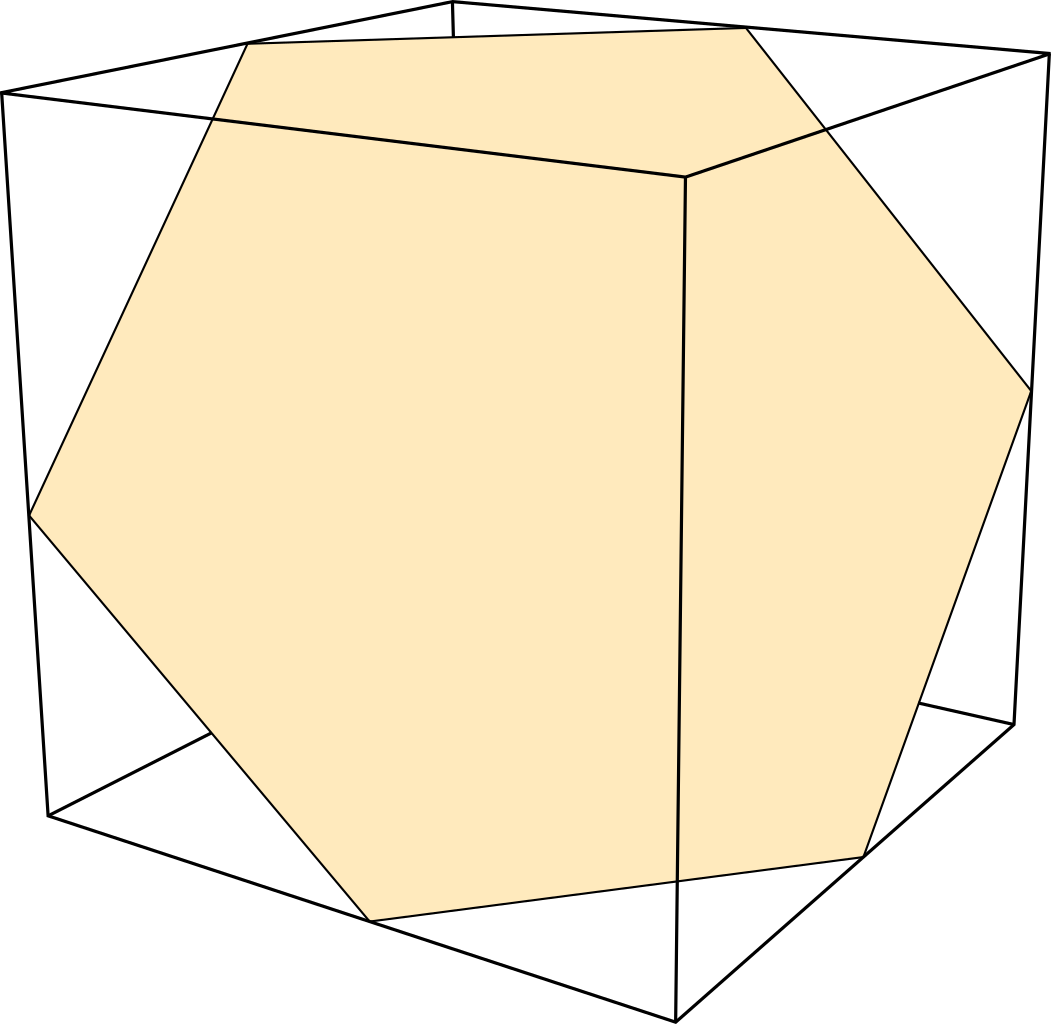} 
        \caption{}
        \label{fig:section_hexagon}
    \end{subfigure}
    \caption{The two combinatorial types of hyperplane sections of the $3$-cube.}
\end{figure}
There are finitely many regions of the sphere for which the combinatorial type stays the same (see \Cref{lemma:subdivision_of_sphere}).  Using this we can parameterize the area of the parallelogram or hexagon with respect to the vector $u$ to construct the radial function of $IP$.  Indeed, as will be shown in the proof of \Cref{thm:intersection_body_semialgbraic}, this can be equivalently written to provide a semialgebraic description of the intersection body.
In particular, if the intersection is a square, then the radial function in a neighborhood of that point will be a constant term over a coordinate variable, e.g. $\frac{4}{z}$. On the other hand, when the intersection is a hexagon, the radial function is a degree two polynomial over $xyz$.
The intersection body is convex as promised by the theory and displayed in \Cref{fig:int_cube}.
We continue with this in \Cref{ex:cube2}.

\end{example}

\begin{lemma}\label{lemma:subdivision_of_sphere}
Let $P$ be a full-dimensional polytope in $\R^d$. Then there exists a central hyperplane arrangement $H$ in $\R^d$ whose maximal open chambers $C$ satisfy the following property. For all $x \in C$, the hyperplane $x^\perp$ intersects a fixed set of edges of $P$ and the polytopes $Q = P \cap x^\perp$ are of the same combinatorial type. 
\end{lemma}
\begin{proof}
Let $x$ be a generic vector of $\R^d$ and consider $Q = P\cap x^\perp$. The vertices of $Q$ are the points of intersection of $x^\perp$ with the edges of $P$. 
Perturbing $x$ continuously, the intersecting edges (and thus the combinatorial type) remain the same, unless the hyperplane $x^\perp$ passes through a vertex $v$ of $P$. 
This happens if and only if $\langle x,v \rangle = 0 $ and thus the set of normal vectors of such hyperplanes is given by 
$
 v^\perp = \{x \in \R^d \mid \langle x, v \rangle = 0 \}.
$
Taking the union over all vertices yields the central hyperplane arrangement
$$
    H = \{ v^\perp \mid v \text{ is a vertex of } P \text{ and $v$ is not the origin}\}.
$$
Then each open region $C$ of the complement of $H$ contains points $x$ such that $x^\perp$ intersects a fixed set of edges of $P$.
\end{proof}

The proof of \Cref{lemma:subdivision_of_sphere} implies that the number of regions we are interested in is the number of chambers of the central hyperplane arrangement $H$. Let $m \!=\! \left(\# \{ v \text{ is a vertex of } P\}/\!\sim\right)$ where $v\sim w$ if $v=\pm w$.
Then we have an upper bound for such a number:
\begin{equation*}
    \sum_{j=0}^d \binom{m}{j} 
\end{equation*}
given by the number of chambers of a generic arrangement \cite[Prop. 2.4]{stanley_introductionhyperplanearrangements}. 

\begin{remark}\label{rem:zonotope}
    We note that there are several ways to view the hyperplane arrangement $H$ in \Cref{lemma:subdivision_of_sphere}.
    For example, since the vertices of $P$ are the normal vectors of the facets of the dual polytope $P^\circ$, we can describe $H$ as the collection of linear hyperplanes which are parallel to facets of $P^\circ$. We also note that $H$ is the normal fan of a zonotope whose edge directions are orthogonal to the hyperplanes of $H$. The fan $\Sigma$ induced by the hyperplane arrangement $H$ is the normal fan of the zonotope 
    \begin{equation*}\label{eq:zonotope}
            Z(P) = \sum_{v \text{ is a vertex of } P} [-v, v].
    \end{equation*}
    We will call this zonotope the \emph{zonotope associated to $P$}. As will be clarified later in \Cref{rem:zonotope_combinatorics}, the dual body of $Z(P)$ plays an important role in the visualization and the combinatorics of the intersection body $IP$.
\end{remark}

\begin{theorem}\label{thm:intersection_body_semialgbraic}
Let $P\subseteq \R^d$ be a full-dimensional polytope containing the origin.  Then $IP$, the intersection body of $P$, is semialgebraic. 
\end{theorem}

\begin{proof}
Fix a region $U=C\cap S^{d-1}$ for an open cone $C$ from \Cref{lemma:subdivision_of_sphere}. Then for every $u\in U$ the hyperplane $u^\perp$ intersects $P$ in the same set of edges.
Let $v$ be a vertex of $Q=P\cap u^\perp$. Then there is an edge $[a,b]$ of $P$ such that $v = [a,b] \cap u^\perp$. This implies that $v = \lambda a + (1-\lambda) b$ for some $\lambda \in (0,1)$ and $\langle v,u\rangle =0$.
From this we get that 
$$\lambda = \frac{\langle b, u \rangle}{\langle b-a, u \rangle}$$
which implies that
$$
v = \frac{\langle b, u \rangle}{\langle b-a, u \rangle}(a-b)+b = \frac{\langle b, u \rangle a - \langle a, u \rangle b}{\langle b-a, u \rangle}.
$$
In this way we express $v$ as a function of $u$ (for fixed $a$ and $b$).
Let $v_1, \ldots, v_n$ be the vertices of $Q$ and let $[a_i,b_i]$ be the corresponding edges of $P$.

We now consider the following triangulation of $Q$:
first, triangulate each facet of $Q$ that does not contain the origin, without adding new vertices (this can always be done e.g. by a regular subdivision using a generic lifting function, cf. \cite[Prop. 2.2.4]{loera_triangulations}). For each $(d-2)$-dimensional simplex $\Delta$ in this triangulation, consider the $(d-1)$-dimensional simplex $\conv(\Delta,0)$ with the origin. This constitutes a triangulation $T = \{\Delta_j : j\in J\}$ of $Q$, in which the origin is a vertex of every simplex. \\
Restricting to $U$, the radial function of the intersection body $IP$ in direction $u$ is the volume of $Q$, and hence given by 
$$\rho_{IP}(u) = \vol (Q) = \sum_{j\in J} \vol(\Delta_j).$$
We can thus compute $\rho_{IP}(u)$ as
$$\rho_{IP}(u) = \sum_{j\in J} \frac{1}{(d-1)!} \left| \det \left(M_j(u)\right) \right|,$$
where
$$M_j(u) = \begin{bmatrix}
v_{i_1}(u) \\
v_{i_2}(u) \\
\vdots \\
v_{i_{d-1}}(u) \\
u
\end{bmatrix} = \begin{bmatrix}
\frac{\langle b_{i_1}, u \rangle a_{i_1} - \langle a_{i_1}, u \rangle b_{i_1}}{\langle b_{i_1}-a_{i_1}, u \rangle} \\
\vdots \\
\frac{\langle b_{i_{d-1}}, u \rangle a_{i_{d-1}} - \langle a_{i_{d-1}}, u \rangle b_{i_{d-1}}}{\langle b_{i_{d-1}}-a_{i_{d-1}}, u \rangle} \\
u
\end{bmatrix}$$
and the row vectors $\{v_{i_1},v_{i_2}, \ldots, v_{i_{d-1}}\}$ (along with the origin) are vertices of the simplex $\Delta_j$ of the triangulation. Therefore, we obtain an expression $\rho_{IP}(u) = \frac{p(u)}{q(u)}$ for some polynomials $p,q\in \R[u_1,\ldots ,u_d]$ without common factors, for $u\in U$. With the same procedure applied to all regions $U_i = C_i \cap S^{d-1}$, for $C_i$ as in \Cref{lemma:subdivision_of_sphere}, we obtain an expression for $\rho\vert_{S^{d-1}}$ that is continuous and piecewise a quotient of two polynomials $p_i,q_i$.
It follows from the definition of the radial function that 
\begin{equation*}
    IP = \left\lbrace x \in \R^d \mid \rho_{IP}\left( x \right) \geq 1 \right\rbrace
    = \left\lbrace x \in \R^d \mid \frac{1}{\|x\|}\rho_{IP}\left( \frac{x}{\|x\|} \right) \geq 1 \right\rbrace.
\end{equation*} 
Notice that for every $j\in J$ we have the following equality:
\begin{equation*}
    \det \left( M_j\left( \frac{x}{\|x\|} \right) \right) = 
    \det \begin{bmatrix}
    v_{i_1}\left( \frac{x}{\|x\|} \right) \\[6pt]
    \vdots \\[6pt]
    v_{i_{d-1}}\left( \frac{x}{\|x\|} \right) \\[6pt]
    \frac{x}{\|x\|}
    \end{bmatrix} =
    \det \begin{bmatrix}
    v_{i_1}\left( x \right) \\[9.5pt]
    \vdots \\[9.5pt]
    v_{i_{d-1}}\left( x \right) \\[9.5pt]
    \frac{x}{\|x\|}
    \end{bmatrix} = 
    \frac{1}{\|x\|} \det\left( M_j\left( x \right)\right)
\end{equation*}
and therefore, if $\frac{x}{\|x\|} \in U$, 
\begin{equation*}
    \rho_{IP}\left( \frac{x}{\|x\|} \right) = \sum_{j\in J} \frac{1}{ (d-1)!} \left| \det \left(M_j\left( \frac{x}{\|x\|} \right)\right) \right| = \frac{1}{\|x\|} \sum_{j\in J} \frac{1}{ (d-1)!} \left| \det \left(M_j\left( x \right)\right)\right| = \frac{p(x)}{\|x\| q(x)}.
\end{equation*}
Because the radial function is a semialgebraic map, by quantifier elimination the intersection body is also semialgebraic.
More explicitly, let $I$ be the set of indices $i$ such that  $\rho_{IP}\big|_{U_i} \neq 0$.
Then we can write the intersection body as 
\begin{align*}
    IP &= \bigcup_{i\in I} \left\lbrace x\in \overline{C}_i \mid \frac{1}{\|x\|^2} \cdot \frac{p_i(x)}{q_i(x)} \geq 1  \right\rbrace \\
    &= \bigcup_{i\in I} \left\lbrace  x\in \overline{C}_i \mid  \| x \|^2 q_i(x) - p_i(x) \leq 0  \right\rbrace.
\end{align*} 
This expression gives exactly a semialgebraic description of $IP$.
\end{proof}

\begin{example}\label{ex:icosahedron1}
Let $P$ be the regular icosahedron in $\R^3$, whose $12$ vertices are all the even permutations of $\left(0, \pm \frac{1}{2}, \pm (\frac{1}{4} \sqrt{5} + \frac{1}{4})\right)$. The associated hyperplane arrangement has $32=12+20$ chambers. The first type of chambers is spanned by five rays and the radial function of $IP$ is given by a quotient of a quartic and a quintic, defined over $\Q(\sqrt{5})$. In the remaining twenty chambers $\rho_{IP}$ is a quintic over a sextic, again with coefficients in $\Q(\sqrt{5})$. This intersection body is the convex set shown in \Cref{fig:int_icosahedron}. We will continue the analysis of $IP$ in \Cref{ex:icosahedron2}.
\end{example}

The theory of intersection bodies assures that the intersection body of a centrally symmetric convex body is again a centrally symmetric convex body, as it happens in \Cref{ex:cube} and in \Cref{ex:icosahedron1}.
On the other hand, given any polytope $P$ (indeed this holds more generally for any convex body) there exists a translation of $P$ such that $IP$ is not convex. This is the content of the next example.

\begin{example}\label{ex:cube_vertex}
Let $P$ be the cube $[-1,1]^3 + (1,1,1)$, so that the origin is a vertex of $P$. The hyperplane arrangement associated to $P$ divides the space in $32$ chambers. In two of them the radial function is $0$. In six regions the radial function has the following shape (up to permutation of the coordinates and sign):
\begin{equation*}
    \rho(x,y,z) = \frac{4}{z}.
\end{equation*}
There are then $18 = 6 + 12$ regions in which the radial function looks like
\begin{equation*}
    \rho(x,y,z) = \frac{2x}{yz} \qquad \hbox{or} \qquad \rho(x,y,z) = \frac{2(x+2z)}{yz}.
\end{equation*}
In the remaining six regions we have
\begin{equation*}
    \rho(x,y,z) = \frac{2(x^2+2xy+y^2+2xz+z^2)}{xyz}.
\end{equation*}
\Cref{fig:cube_vertex} shows two different points of view of $IP$, which is in particular not convex.
\end{example}

\begin{figure}[h]
    \centering
    \begin{subfigure}{0.49\textwidth}
    \centering
    \includegraphics[width=0.58\textwidth]{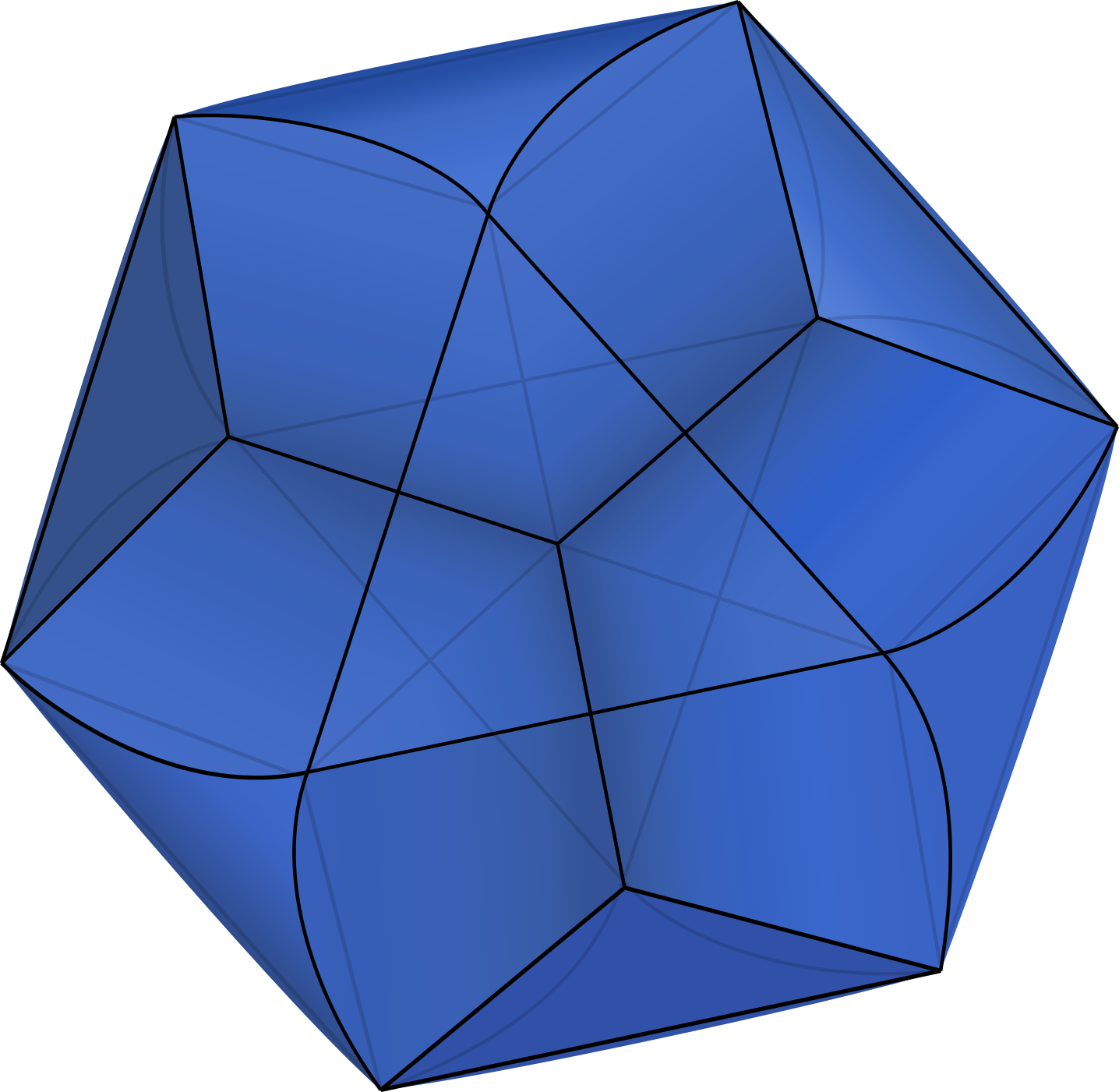}
    \caption{}
    \label{fig:cube_vertex_1}
    \end{subfigure}
    \begin{subfigure}{0.49\textwidth}
    \centering
    \includegraphics[width=0.76\textwidth]{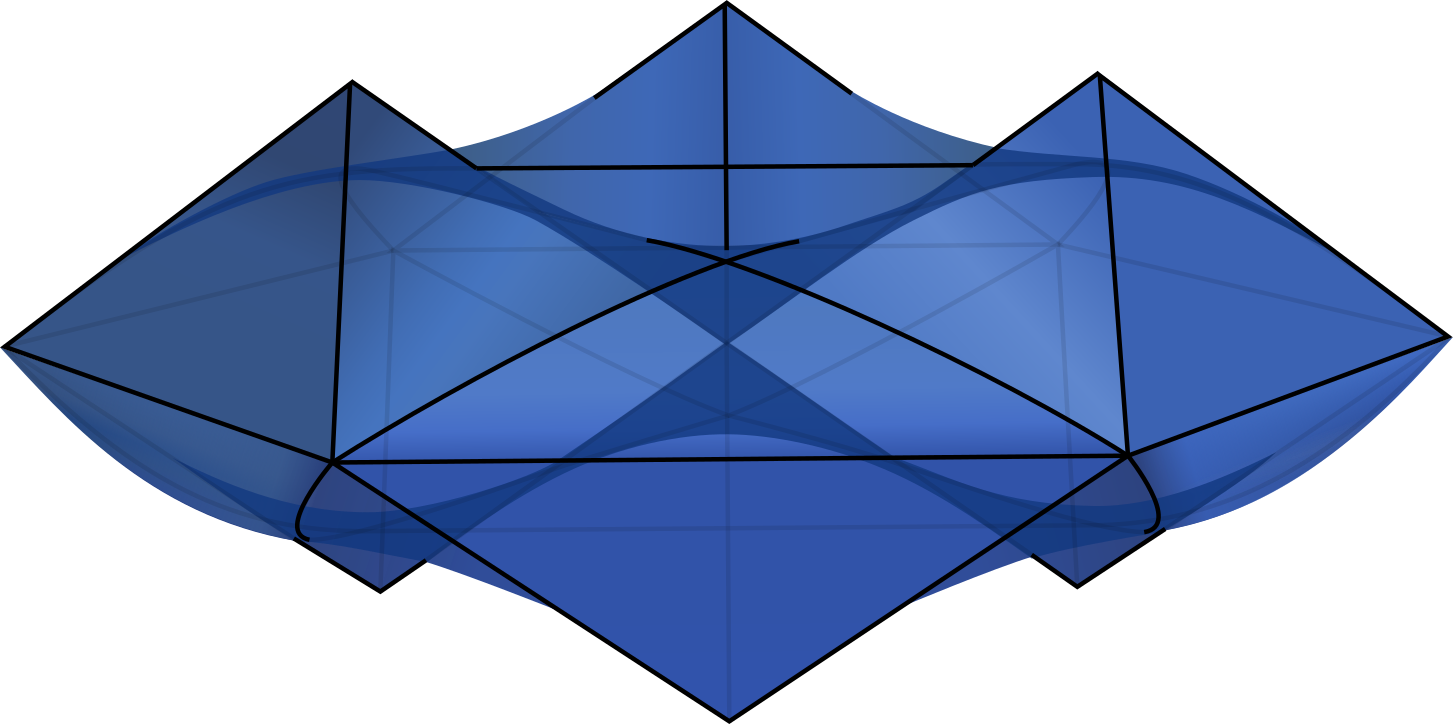}
    \caption{}
    \label{fig:cube_vertex_2}
    \end{subfigure}
    \caption{The intersection body of the cube in \Cref{ex:cube_vertex} from two different points of view.}
    \label{fig:cube_vertex}
\end{figure}

\section{Non-convex Intersection Bodies}\label{sec:non-convex-IP}

The proof of \Cref{thm:intersection_body_semialgbraic} relies on the fact that the origin is in the polytope. However, if the origin is not contained in $P$, we can still find a semialgebraic description of $IP$ by adjusting how we compute the volume of $P\cap u^\perp$.  The remainder of this section will be dedicated to proving this.

\begin{lemma}\label{lem:triangles}
Let $P\subset \R^d$ be a full-dimensional polytope, and let $\mathcal{F}$ be the set of its facets. Let $p$ be a point outside of $P$. For each face $F\in \mathcal{F}$, let $\Hat{F}$ denote the set $\operatorname{conv}(F\cup \{p\})$. Then the following equality holds:
$$\vol(P) = \sum_{F\in \mathcal{F}} \sgn(F)\vol(\Hat{F})$$
where $\sgn(F)$ is $1$ if $P$ and $p$ belong to the same halfspace defined by $F$, and $-1$ otherwise. 
\end{lemma} 

\begin{proof}
Let $\Hat{P} = \operatorname{conv}(P\cup\{p\})$ and denote by $\mathcal{F}^+_p$ the set of facets $F$ of $P$ for which the halfspace defined by $F$ containing $P$ also contains $p$, possibly on its boundary. Let $\mathcal{F}^-_p = \mathcal{F}\setminus \mathcal{F}^+_p$. 

First we will show that $\Hat{P} = \bigcup_{F\in \mathcal{F}^+_p} \Hat{F}$. The inclusion $\bigcup_{F\in \mathcal{F}^+_p} \Hat{F}\subseteq \Hat{P}$ follows immediately from convexity. To see the opposite direction, let $q\in \Hat{P}$ and consider $r$ to be the ray starting at $p$ and going through $q$. Either $r$ intersects $P$ only along its boundary, or there are some intersection points also in the interior of $P$. In the first case $r\cap P \subset F$ and so $q\in \Hat{F}$ for some face $F$, that by convexity must be in $\mathcal{F}^+_p$.
On the other hand, if the ray $r$ intersects the interior of the polytope $P$, denote by $a$ the farthest among the intersection points:
\begin{equation*}
    \|a-p\| = \max \{\|\alpha-p\| \mid \alpha\in P\cap r\}.
\end{equation*}
Let $F_a$ be a facet containing $a$. Then, $q$ is contained in the convex hull of $F_a \cup \{p\}$, i.e. $\Hat{F}_a$.  From the definition of $a$ it follows that the halfspace defined by $F_a$ containing $p$ must also contain $P$, so $F_a\in \mathcal{F}^+_p$ and our statement holds.

Next, we will show that $\bigcup_{F\in \mathcal{F}^-_p} \Hat{F} = \overline{\Hat{P} \setminus P}$. The pyramid $\Hat{F}$ is contained in the closed halfspace defined by $F$ which contains $p$. By the definition of $\mathcal{F}^-_p$, this halfspace does not contain $P$ thus $\Hat{F}\cap P = F$. Also, $\Hat{F}~\subseteq~\Hat{P} $ so we have that $\Hat{F}\subseteq~\overline{\Hat{P}\setminus P}$ and hence  $\bigcup_{F\in \mathcal{F}^-_p} \Hat{F}~\subseteq~\overline{\Hat{P} \setminus P}$. 
Conversely, let $q\in \overline{\Hat{P}\setminus P}$. If $q~=~p$ we are done, so assume $q \neq p$. Then, $q = \lambda p + (1-\lambda) b $ for some $b\in P$, $\lambda \in [0,1)$. Let $a$ be the point at which the segment from $p$ to $b$ first intersects the boundary of $P$, i.e.
\begin{equation*}
    \|a-p\| = \min \{\|\alpha-p\| \mid \alpha\in P, \alpha = tp+(1-t)b \hbox{ for } t\in [0,1)\}.
\end{equation*}
Then by construction there exists a facet $F_a \in \mathcal{F}^-_p$ containing $a$, such that $q\in\Hat{F}_a$.
Thus, we have that $$ \vol(\bigcup_{F\in \mathcal{F}^+_p}\Hat{F}) = \vol(\Hat{P}) = \vol(\Hat{P}\setminus P) + \vol(P) = \vol(\bigcup_{F\in \mathcal{F}^-_p}\Hat{F}) + \vol(P).$$ 
If $F_1 \neq F_2$ and $F_1, F_2 \in \mathcal{F}^+_p$ or $F_1, F_2 \in \mathcal{F}^-_p$, then the volume of $\Hat{F}_1\cap \Hat{F}_2$ is zero, therefore
$$\sum_{F\in \mathcal{F}^+_p}\vol(\Hat{F}) = \sum_{F\in \mathcal{F}^-_p}\vol(\Hat{F}) + \vol (P)$$ and the claim follows.
\end{proof}

\begin{theorem}\label{theorem:IP_semialgebraic}
Let $P\subset \R^d$ be a full-dimensional polytope.  Then $IP$, the intersection body of $P$, is semialgebraic. 
\end{theorem}

\begin{proof}
What remains to be shown is that $IP$ is semialgebraic in the case when the origin is not contained in $P$, and hence it is not contained in any of its sections $Q = P \cap u^\perp$. From \Cref{lem:triangles}, with $p = 0 \in \R^d$ we have that 
$$
\vol(Q) = \sum_{F \text{ facet of } Q} \sgn(F) \vol(\Hat{F})
$$ 
where $\Hat{F}$ is the convex hull of $F$ and the origin. Let $T_F = \{\Delta_j : j\in J_F\}$
be a triangulation of $F$. We can calculate as in the proof of \Cref{thm:intersection_body_semialgbraic}
$$
\vol(\Hat{F)} = \sum_{j\in J_F} \frac{1}{(d-1)!} |\det M_j|
$$ 
where $M_j$ is the matrix whose rows are the vertices of the simplex $\Delta_j \in T_F$ and $u$.
We then follow the remainder of the proof of \Cref{thm:intersection_body_semialgbraic} to see that the intersection body is semialgebraic.
\end{proof}

\section{The Algorithm}
\label{section:algorithm}

The proofs from \Cref{thm:intersection_body_semialgbraic} and \Cref{theorem:IP_semialgebraic} lead to an algorithm to compute the radial function of the intersection body of a polytope.
In this section, we describe the algorithm.
By \Cref{rem:zonotope}, the regions $C$ in which $\rho(x)|_C = \frac{p(x)}{\| x \|^2 q(x)}$ for fixed polynomials $p(x)$ and $q(x)$ are defined by the normal fan of the zonotope $Z(P)$. First, we compute the radial function for each of these cones individually, by applying \Cref{alg:rho-for-region}. \\

\begin{algorithm}[H]
\caption{Computing the radial function for a fixed region $C$}\label{alg:rho-for-region}
\begin{algorithmic}[1]
	\item[]
	\REQUIRE A full-dimensional polytope $P$ in $\R^d$.
	\REQUIRE A maximal open cone $C$ of the normal fan of $Z(P)$.
	\ENSURE The radial function $\rho(x)$ of the intersection body $IP$ restricted to $C$.
	\STATE  Let $\mathcal F$ be the collection of facets of $P$ such that for all $u\in U = C \cap S^{d-1}$ and $F\!\in\! \mathcal F$ holds: $\dim(F\cap u^\perp) = \dim(P)-2$ and $0 \not \in F$.
	\STATE Let $Q = P \cap u^\perp\!, u\!\in\! U$. Triangulate $F\cap u^\perp$ for $F \in \mathcal{F}$, i.e. all facets of $Q$ not contai- \\ ning the origin. Let $\mathcal{T}$ be the collection of all maximal cells of these triangulations. \label{alg-step:triangulation}
	\FOR {each cell $\Delta \in \mathcal T$}
	\STATE Let $v_1, \dots, v_{d-1}$ be the vertices of $\Delta$ in orientation-preserving order.
	\STATE For $i=1,\dots,d-1$, let $e_i=\conv(a_i,b_i)$ be the edge of $P$ such that $e_i\cap u^\perp = v_i$.  
	\STATE Let $x = (x_1,\dots,x_n)$ be a vector with indeterminates $x_1,\dots, x_n$. Let $M_\Delta$ be the $(d\times d)$-matrix with $i$th row $\frac{\langle b_{i}, x\rangle a_{i} - \langle a_{i}, x\rangle b_{i}}{\langle b_{i}-a_{i}, x \rangle}$ and last row $x$.
	\IF { $ \conv(\mathbf 0, \Delta)$ intersects the interior of $P$}
	\STATE Define $\sgn(\Delta) = 1$
	\ELSE {}
	\STATE Define $\sgn(\Delta) = -1$
	\ENDIF
	\ENDFOR
	\RETURN $\frac{1}{\| x \| ^2 (d-1)!}\sum_{\Delta\in \mathcal T} \sgn(\Delta)\det(M_\Delta)$
\end{algorithmic} 
\end{algorithm} 

This algorithm has as output the rational function  $\rho(x)|_C = \frac{p(x)}{\| x \|^2 q(x)}$.
Iterating over all regions yields the final \Cref{alg:global_rho}.\\

\begin{algorithm}[H]
\caption{Computing the radial function of $IP$}\label{alg:global_rho}
\begin{algorithmic}[1]
	\item[]
	\REQUIRE A full-dimensional polytope $P$ in $\R^d$.
	\ENSURE The radial function $\rho(x)$ of the intersection body $IP$.
	\STATE Let $\Sigma$ be the polyhedral fan from \Cref{rem:zonotope}.
	\FOR {each maximal open region $C$ of $\Sigma$}
	 \STATE Compute $\rho|_C$ via \Cref{alg:rho-for-region}.
	\ENDFOR
	\RETURN $\left( \frac{1}{\| x \| ^2 (d-1)!}\sum_{\Delta\in \mathcal T} \sgn(\Delta)\det(M_\Delta), \ C \right)$ for $C \in \Sigma$
\end{algorithmic} 
\end{algorithm}

An implementation of these algorithms for \texttt{SageMath 9.2} \cite{sagemath} and \texttt{Oscar 0.8.2-DEV} \cite{OSCAR} can be found in \url{https://mathrepo.mis.mpg.de/intersection-bodies}. Note that in step \ref{alg-step:triangulation} of \Cref{alg:rho-for-region}, the implementation uses a regular subdivision of the facets of the polytope $Q$ by lifting the vertices $v_1, \dots, v_m$ along the moment curve $(t^1,\dots,t^m)$ with $t=3$.

\section{Algebraic Boundary and Degree Bound}
\label{section:algebraic_boundary}

In order to study intersection bodies from the point of view of real algebraic geometry we need to introduce our main character for this section, the algebraic boundary. For more on the algebraic boundary we refer the reader to \cite{sinnAlgebraicBoundariesConvex2015}.
\begin{definition}
    Let $K$ be any compact subset in $\R^d$, then its \emph{algebraic boundary} $\partial_a K$ is the $\R$-Zariski closure of the Euclidean boundary $\partial K$.
\end{definition}

Knowing the radial function of a convex body $K$ implies knowing its boundary. In fact, when $0\in \operatorname{int} K$ then $x\in \partial K$ if and only if $\rho_K (x) = 1$ (see \Cref{rmk:discontinuity_radial} for the other cases).
Therefore, using the same notation as in the proof of \Cref{thm:intersection_body_semialgbraic}, we can observe that the algebraic boundary of the intersection body of a polytope is contained in the union of the varieties $\mathcal{V}\left(\| x \|^2 q_i(x) - p_i(x) \right)$.
Indeed, we actually know more: as will be proven in \Cref{prop:degree}, the $p_i$'s are divisible by the polynomial $\|x\|^2$, and hence
\begin{equation*}
    \partial_a IP = \bigcup_{i\in I} \mathcal{V}\left(q_i(x) - \frac{p_i(x)}{\|x\|^2} \right)
\end{equation*}
because of the assumption made in the proof of \Cref{thm:intersection_body_semialgbraic} that $p_i,q_i$ do not have common components.  That is, these are exactly the irreducible components of the boundary of $IP$.

\begin{remark}\label{rmk:discontinuity_radial}
As anticipated in \Cref{section:IP_semialgebraic} there may be difficulties when computing the boundary of $IP$ in the case where the origin is not in the interior of the polytope $P$. In particular, $x$ is a discontinuity point of the radial function of $IP$ if and only if $x^\perp$ contains a facet of $P$. Therefore $\rho_{IP}$ has discontinuity points if and only if the origin lies in the union of the affine linear spans of the facets of $P$. In this case, there are finitely many rays where the radial function is discontinuous and they belong to $\R^d \setminus \left( \cup_{i\in I} C_i \right)$, i.e. to the hyperplane arrangement $H$. If $d=2$, these rays disconnect the space, and this implies that we loose part of the (algebraic) boundary of $IP$: to the set $\{x\in\R^d \mid \rho_{IP}(x) = 1\}$ we need to add segments from the origin to the boundary points in the direction of these rays. However, in higher dimensions the discontinuity rays do not disconnect $\R^d$ so $\{x \in \R^d \mid \rho_{IP}(x) = 1\}$ approaches the region where the radial function is zero continuously except for these finitely many directions.  Therefore there are no extra components of the boundary of $IP$.
\end{remark}

\begin{example}[Continuation of \Cref{ex:cube}, cf. \Cref{fig:int_cube}]\label{ex:cube2}
Starting from the radial function of the intersection body of the $3$-cube $P$, computed using \Cref{alg:rho-for-region}, we can recover the equations of its algebraic boundary. The Euclidean boundary of this convex body is divided in $14$ regions. Among them, $6$ arise as the intersection of a convex cone spanned by $4$ rays with a hyperplane; they constitute facets, i.e. flat faces of dimension $d-1$, of $IP$. For example the facet exposed by the vector $(1,0,0)$ is the intersection of $z = 4$ with the convex cone 
\begin{equation*}
    \overline{C}_1 = \co\{(1, 0, 1), (-1, 0, 1), (0, 1, 1), (0, -1, 1)\}.
\end{equation*} In other words, the variety $\mathcal{V}(z-4)$ is one of the irreducible components of $\partial_a IP$.
The remaining $8$ regions are spanned by $3$ rays each, and the polynomial that defines the boundary of $IP$ is a cubic, such as
\begin{equation*}
    2 x y z - 2 x^2 - 4 x y - 2 y^2 - 4 x z + 4 y z - 2 z^2
\end{equation*}
in the region 
\begin{equation*}
    \overline{C}_2 = \co\{(0, 1, 1), (-1, 1, 0), (-1, 0, 1)\}.
\end{equation*}
These cubics are in fact, up to a change of coordinates, the algebraic boundary of a famous spectrahedron: the elliptope \cite{laurent:elliptope}.
Hence $\partial_a IP$ is the union of $14$ irreducible components, six of degree $1$ and eight of degree $3$.
\end{example}

\begin{remark}
In \cite{plaumann2021families} the authors introduce the notion of \emph{patches} of a semialgebraic convex body, with the purpose of mimicking the faces of a polytope. In the case of intersection bodies of polytopes, it is tempting to think that each region of \Cref{lemma:subdivision_of_sphere} corresponds to a patch. Indeed, this happens, for example, for the centered $3$-cube in \Cref{ex:cube2}. On the other hand, if $P = [-1,1]^3+(0,0,1)$ then there are $4$ regions that define the same patch of the algebraic boundary of $IP$; therefore there is, unfortunately, no one-to-one correspondence between regions and patches.
\end{remark}

\begin{prop}\label{prop:degree}
Using the notation of \Cref{lemma:subdivision_of_sphere} and \Cref{theorem:IP_semialgebraic}, fix a chamber $C$ of $H$ and let $Q = P\cap u^{\perp}$ for some $u\in U = C \cap S^{d-1}$. Then the polynomial $\|x\|^2 = x_1^2+\ldots +x_d^2$ divides $p(x)$ and
\begin{equation*}\label{eq:bound_vertices_Q}
    \deg \left( q(x) - \frac{p(x)}{\|x\|^2} \right) \leq f_0(Q).
\end{equation*}
\end{prop}
\begin{proof}
For the fixed region $C$, let $T$ be a triangulation of $Q$ with simplices indexed by $J$. Then the volume of $Q$ is given by 
\begin{equation*}
    \frac{p(x)}{q(x)} =  \frac{1}{(d-1)!} \sum_{j\in J} \left| \det \left(M_j\left( x \right)\right)\right|,
\end{equation*}
where $M_j$ is the matrix as in the proof of \Cref{thm:intersection_body_semialgbraic}.  Notice that for each $M = M_j$, we can rewrite the determinant to factor out a denominator (we also write for simplicity $\Delta = \Delta_j$):
\begin{align*} 
    \det(M(x)) &= \sum_{\sigma \in S_d} \sgn(\sigma) \prod_{i=1}^d M_{i\sigma(i)} \\
    &= \sum_{\sigma \in S_d} \sgn(\sigma) x_{\sigma(d)} \prod_{i=1}^{d-1} \frac{\langle b_{i}, u \rangle a_{i\sigma(i)} - \langle a_{i}, u \rangle b_{i\sigma(i)}}{\langle b_{i}-a_{i}, u \rangle} \\
    &= \prod_{i=1}^{d-1} \frac{1}{\langle b_{i}-a_{i}, u \rangle} \sum_{\sigma \in S_d} \sgn(\sigma) x_{\sigma(d)} \prod_{i=1}^{d-1} \left(\langle b_{i}, u \rangle a_{i\sigma(i)} - \langle a_{i}, u \rangle b_{i\sigma(i)}\right)\\
     &= \left( \prod_{\substack{v_i \in \Delta \\ \text{vertex}}} \frac{1}{\langle b_{i}-a_{i}, x \rangle} \right) \cdot \det \left (\hat{M}\left ( x \right )\right )
\end{align*}
where 
$$
\hat{M}(x) = 
    \begin{bmatrix}
    \vdots \\
    \langle b_{i}, x \rangle a_{i} - \langle a_{i}, x \rangle b_{i} \\
    \vdots \\
    x
    \end{bmatrix}
$$
and the determinant of $\hat{M}(x)$ is a polynomial of degree $d$ in the $x_i$'s. 
Note that if we multiply $\hat{M}(x)\cdot x$ we obtain the vector $(0,\ldots ,0,x_1^2+\ldots+x_d^2)$. Hence if $x_1^2+\ldots+x_d^2 = 0$, then $\hat{M}(x)\cdot x = 0$, i.e. the kernel of $\hat{M}(x)$ is non-trivial and thus $\det \hat{M}(x) = 0$. This implies the containment of the complex varieties $\mathcal{V}(\|x\|^2) \subseteq \mathcal{V}(\det \hat{M}(x))$ and therefore the polynomial $x_1^2+\ldots+x_d^2$ divides the polynomial $\det \hat{M}(x)$.
When we sum over all the simplices in the triangulation $T$ we obtain that
\begin{align*}
    q(x) &= (d-1)! \left( \prod_{\substack{v_i \in \Delta \\ \hbox{vertex}}} \frac{1}{\langle b_{i}-a_{i}, x \rangle} \right) \cdot \left( \prod_{\substack{v_i \notin \Delta \\ \hbox{vertex}}} \frac{1}{\langle b_{i}-a_{i}, x \rangle} \right) \\
    &= \prod_{\substack{v_i \in Q \\ \hbox{vertex}}} \frac{1}{\langle b_{i}-a_{i}, x \rangle}
\end{align*}
and
\begin{equation*}
    p(x) = \sum_{j\in J} \left(\left| \det \left(\hat{M}\left( x \right)\right)\right| \cdot \prod_{\substack{v_i \notin \Delta \\ \hbox{vertex}}} \frac{1}{\langle b_{i}-a_{i}, x \rangle} \right).
\end{equation*}
Hence $\deg q \leq f_0(Q)$ and $\deg p \leq f_0(Q) + 1$, so the claim follows.
\end{proof}

Notice that generically, meaning for the generic choice of the vertices of $P$, the bound in \Cref{prop:degree} is attained,
 because $p$ and $q$ will not have common factors.

\begin{theorem}\label{cor:degree}
Let $P\subset \R^d$ be a full-dimensional polytope with $f_1(P)$ edges. Then the degrees of the irreducible components of the algebraic boundary of $IP$ are bounded from above by
\begin{equation*}
    f_1(P) - (d-1).
\end{equation*}
\end{theorem}
\begin{proof}
We want to prove that $f_0(Q) \leq f_1(P) - (d-1)$, for every $Q = P \cap u^\perp$, $u \in S^{d-1} \setminus H$. 
By definition, every vertex of $Q$ is a point lying on an edge of $P$, so trivially $f_0(Q) \leq f_1(P)$. We want to argue now that it is impossible to intersect more than $f_1(P) - (d-1)$ edges of $P$ with our hyperplane $\mathcal{H}= u^\perp$. If the origin is one of the vertices of $P$, then all the edges that have the origin as a vertex give rise only to one vertex of $Q$: the origin itself. There are at least $d$ such edges, because $P$ is full-dimensional, and so $f_0(Q) \leq f_1(P) - (d-1)$. \\

Suppose now that the origin is not a vertex of $P$, then $\mathcal H$ does not contain vertices of $P$. It divides $\R^d$ in two half spaces $\mathcal H_+$ and $\mathcal H_-$, and so it divides the vertices of $P$ in two families of $k$ vertices in $\mathcal H_+$ and $\ell$ vertices in $\mathcal H_-$. Either $k$ or $\ell$ are equal to $1$, or they are both greater than one. In the first case let us assume without loss of generality that $k=1$, i.e. there is only one vertex $v_+$ in $\mathcal H_+$. Then pick one vector $v_-$ in $\mathcal H_-$: because $P$ is a full-dimensional polytope, there are at least $d$ edges of $P$ with $v_-$ as a vertex. Only one of them may connect $v_-$ to $v_+$ and therefore the other $d-1$ edges must lie in $\mathcal H_-$. This gives $f_0(Q) \leq f_1(P) - (d-1)$. \\

On the other hand, let us assume that $k,\ell \geq 2$. 
Then there is at least one edge in $\mathcal H_+$ and one edge in $\mathcal H_-$. If $d=3$ these are the $d-1$ edges that do not intersect the hyperplane. For $d>3$ we reason as follows. Suppose that $\mathcal H$ intersects a facet $F$ of $P$. Then it cannot intersect all the facets of $F$ (i.e. a ridge of $P$), otherwise we would get $F\subset \mathcal H$ which contradicts the fact that $\mathcal H$ does not intersect vertices of $P$. So there exists a ridge $F'$ of $P$ that does not intersect the hyperplane; it has dimension $d-2\geq 2$ and therefore it has at least $d-1$ edges. Therefore $$f_0(Q) \leq f_1(P) - (d-1).$$
\end{proof}

\begin{cor}\label{cor:degree_cc}
In the hypotheses of \Cref{cor:degree}, if $P$ is centrally symmetric and centered at the origin, then we can improve the bound with 
\begin{equation*}
    \frac{1}{2} \left( f_1(P) - (d-1) \right).
\end{equation*}
\end{cor}
\begin{proof}
We already know that for each chamber $C_i$ from \Cref{lemma:subdivision_of_sphere}, 
the degree of the corresponding irreducible component is bounded by the degree of the polynomial $q_i$. 
This follows from the construction of $p_i$ and $q_i$ in the proof of \Cref{thm:intersection_body_semialgbraic}.  Specifically, the determinant which gives $p_i/q_i$ comes with the product of $d-1$ rational functions, with linear numerator and denominators, and one linear term.  Thus $\deg p_i = \deg q_i + 1$ which implies that $\deg \frac{p_i}{||x||^2} < \deg q_i$.
By definition these polynomials are obtained as the least common multiple of objects with shape
\begin{equation*}
    \prod_{\substack{v_k \in \Delta_j \\ \hbox{vertex}}} \frac{1}{\langle b_{k}-a_{k}, x \rangle}.
\end{equation*}
If $P$ is centrally symmetric, so is $Q$, and therefore we have the vertex belonging to the edge $[a_k,b_k]$ and also the vertex belonging to the edge $[-a_k,-b_k]$. When computing the least common multiple, these two vertices produce the same factor, up to a sign, and therefore they count as the same linear factor of $q_i$. Hence for every $i$
\begin{equation*}
    \deg q_i(x) \leq \frac{f_0(Q)}{2} \leq   
    \frac{1}{2} \left( f_1(P) - (d-1) \right).
\end{equation*}
\end{proof}

\begin{example}\label{ex:tetra}
Let $P$ be the tetrahedron in $\R^3$ with vertices $ (-1, -1, -1), (-1, 1, 1), (1, -1, 1)$, $(1, 1, -1)$.
The associated hyperplane arrangement coincides with the one associated to the cube in \Cref{ex:cube2}, so it has $14$ chambers that come in two families. 
\begin{figure}[!ht]
    \centering
    \begin{subfigure}{0.32\textwidth}
    \centering
    \includegraphics[width=0.75\textwidth]{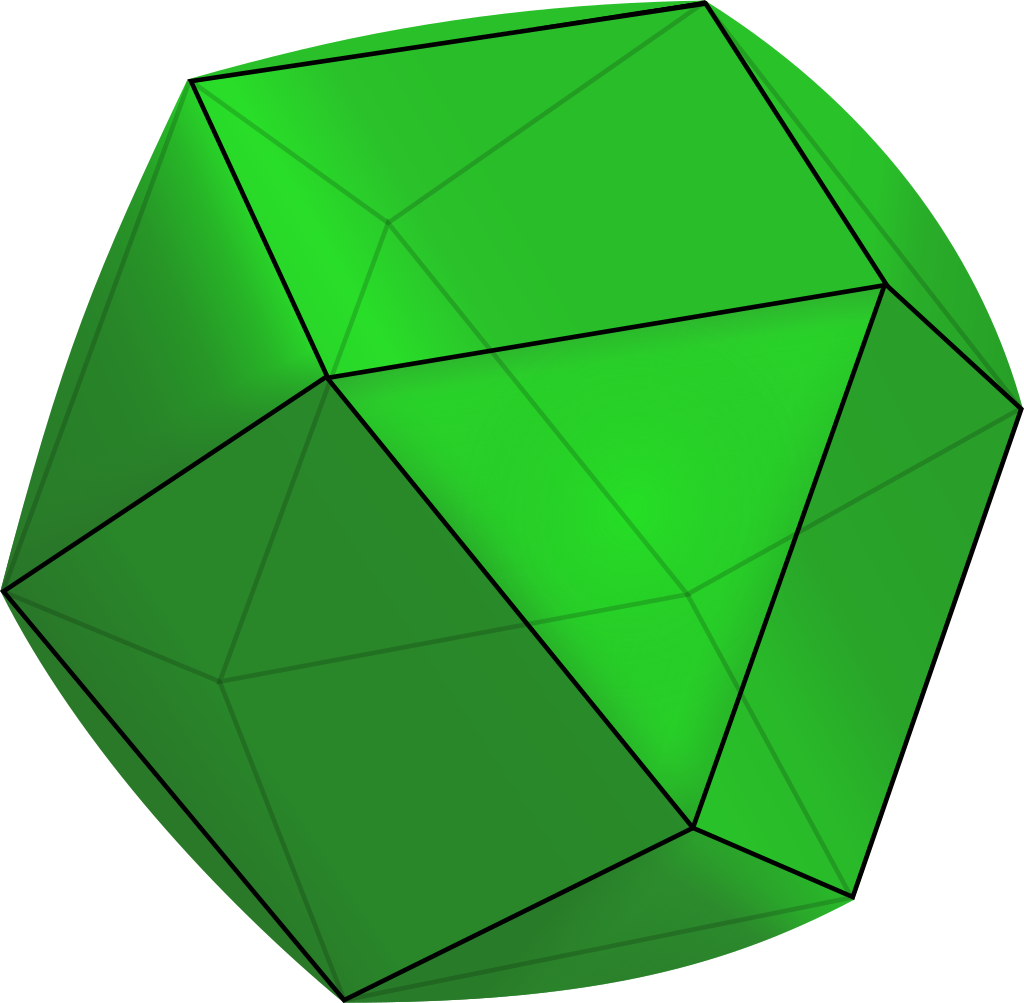}
    \caption{}
    \label{fig:int_cube}
    \end{subfigure}
    \begin{subfigure}{0.32\textwidth}
    \centering
    \includegraphics[width=0.75\textwidth]{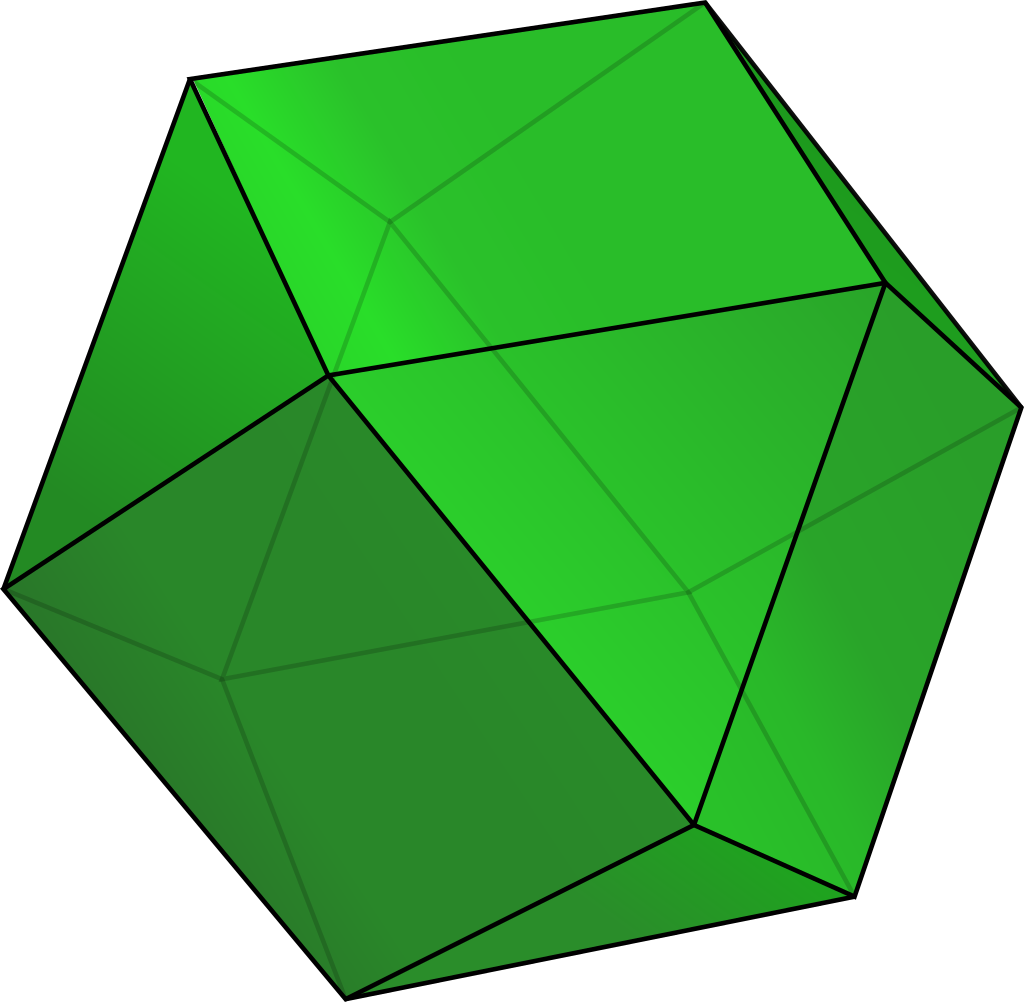}
    \caption{}
    \label{fig:zonotope}
    \end{subfigure}
    \begin{subfigure}{0.32\textwidth}
    \centering
    \includegraphics[width=0.75\textwidth]{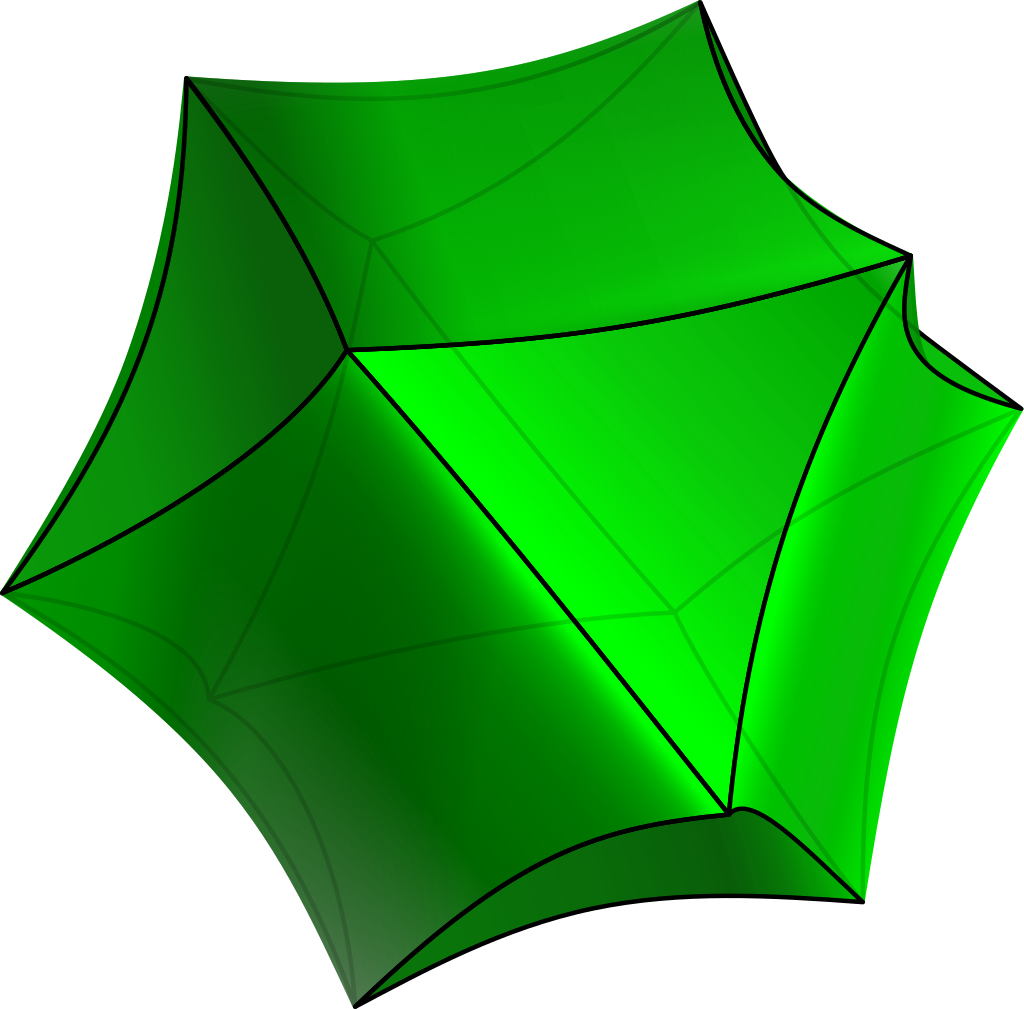}
    \caption{}
    \label{fig:int_tetra}
    \end{subfigure}
    \caption{Left: the intersection body of the cube in \Cref{ex:cube2}. Right: the intersection body of the tetrahedron in \Cref{ex:tetra}. Center: the dual body of the zonotope $Z(P)$ associated to both the cube and the tetrahedron. Such a polytope reveals the structure of the boundary divided into regions of these two intersection bodies.}
    \label{fig:tetra_zonot}
\end{figure}
The first one consists of cones spanned by four rays, such as $\overline{C}_1$ (see \Cref{ex:cube2}). The polynomial that defines the boundary of $IP$ in this region is a quartic, namely
\begin{equation*}
    q_2(x,y,z) - \frac{p_2(x,y,z)}{\|(x,y,z)\|^2} = (x + z) (x - z) (y + z) (y - z) - 2 ( x^2 + y^2 - z^2) z.
\end{equation*}
On the other hand the cones of the second family are spanned by three rays: here the section of $P$ is a triangle and the equation of the boundary if $IP$ is a cubic. An example is the cone $\overline{C}_2$ with the polynomial
\begin{equation*}
    q_1(x,y,z) - \frac{p_1(x,y,z)}{\|(x,y,z)\|^2} =
    (x - y) (x - z) (y + z) + ( x - y - z)^2.
\end{equation*}
Note that this region furnishes an example in which the bounds given in \Cref{prop:degree} and \Cref{cor:degree} are attained.
\end{example}

\begin{remark} \label{rem:zonotope_combinatorics}
\Cref{rem:zonotope} together with \Cref{prop:degree} implies that the structure of the irreducible components of the algebraic boundary of $IP$ is strongly connected with the face lattice of the dual of the zonotope $Z(P)$. More precisely, in the generic case, the lattice of intersection of the irreducible components is isomorphic to the face lattice of the dual polytope $Z(P)^\circ$.
Thus, a classification of ``combinatorial types'' of such intersection bodies is analogous to the classification of zonotopes / hyperplane arrangements / oriented matroids. It is however worth noting, that the same zonotope can be associated to two polytopes $P_1$ and $P_2$ which are not combinatorially equivalent.
One example of this instance is a pair of polytopes such that $P_1 = \conv(v_1, \dots, v_n)$ and $P_2 = \conv(\pm v_1, \dots, \pm v_2)$, as can be seen in \Cref{fig:tetra_zonot} for the cube and the tetrahedron. To have a better overview over the structure of the boundary of $IP$, one strategy is to use the Schlegel diagram of $Z(P)^\circ$. We label each maximal cell by the degree of the polynomial that defines the corresponding irreducible component of $\partial IP$, as can be seen in \Cref{fig:schlegel_ico,fig:schlegel_diagrams}.
\end{remark}

\begin{example}[Continuation of \Cref{ex:icosahedron1}, cf. \Cref{fig:int_icosahedron}]\label{ex:icosahedron2}
Let $P$ be the regular icosahedron. In the $12$ regions which are spanned by five rays, the polynomial that defines the boundary of $IP$ has degree $5$ and it looks like
\begin{gather*}
    ( (\sqrt{5} x + \sqrt{5} y - x + y)^2 - 4z^2) ( (\sqrt{5} x + x + 2 y)^2 - (\sqrt{5} z - z)^2) y +\\
    8 \sqrt{5} x^3 y + 68 \sqrt{5} x^2 y^2 + 72 \sqrt{5} x y^3 + 20 \sqrt{5} y^4 - 40 \sqrt{5} x y z^2 - 20 \sqrt{5} y^2 z^2 + 4 \sqrt{5} z^4 +\\
    8 x^3 y + 164 x^2 y^2 + 168 x y^3 + 44 y^4 - 8 x^2 z^2 - 72 x y z^2 - 44 y^2 z^2 + 12 z^4.
\end{gather*}
In the other $20$ regions spanned by three rays, $\partial IP$ is the zero set of a sextic polynomial with the following shape
\begin{gather*}
    ( (\sqrt{5}  x + x + 2  y )^2 - ( \sqrt{5}  z - z)^2 ) ( (\sqrt{5}  y - 2  x -  y )^2 - ( \sqrt{5}  z - z)^2 ) x y +
    20  \sqrt{5}  x^4  y - \\
    20  \sqrt{5}  x^2  y^3 - 4  \sqrt{5}  x  y^4 + 4  \sqrt{5}  y^5 - 4  \sqrt{5}  x^3  z^2 - 60  \sqrt{5}  x^2  y  z^2 - 12  \sqrt{5}  x  y^2  z^2 + 12  \sqrt{5}  x  z^4 +  44  x^4  y - \\
    8  x^3  y^2 - 44  x^2  y^3 + 12  x  y^4 + 12  y^5 - 12  x^3  z^2 - 156  x^2  y  z^2 - 60  x  y^2  z^2 - 8  y^3  z^2 + 28  x  z^4.
\end{gather*}
\begin{figure}
    \centering
    \includegraphics[width=0.37\textwidth]{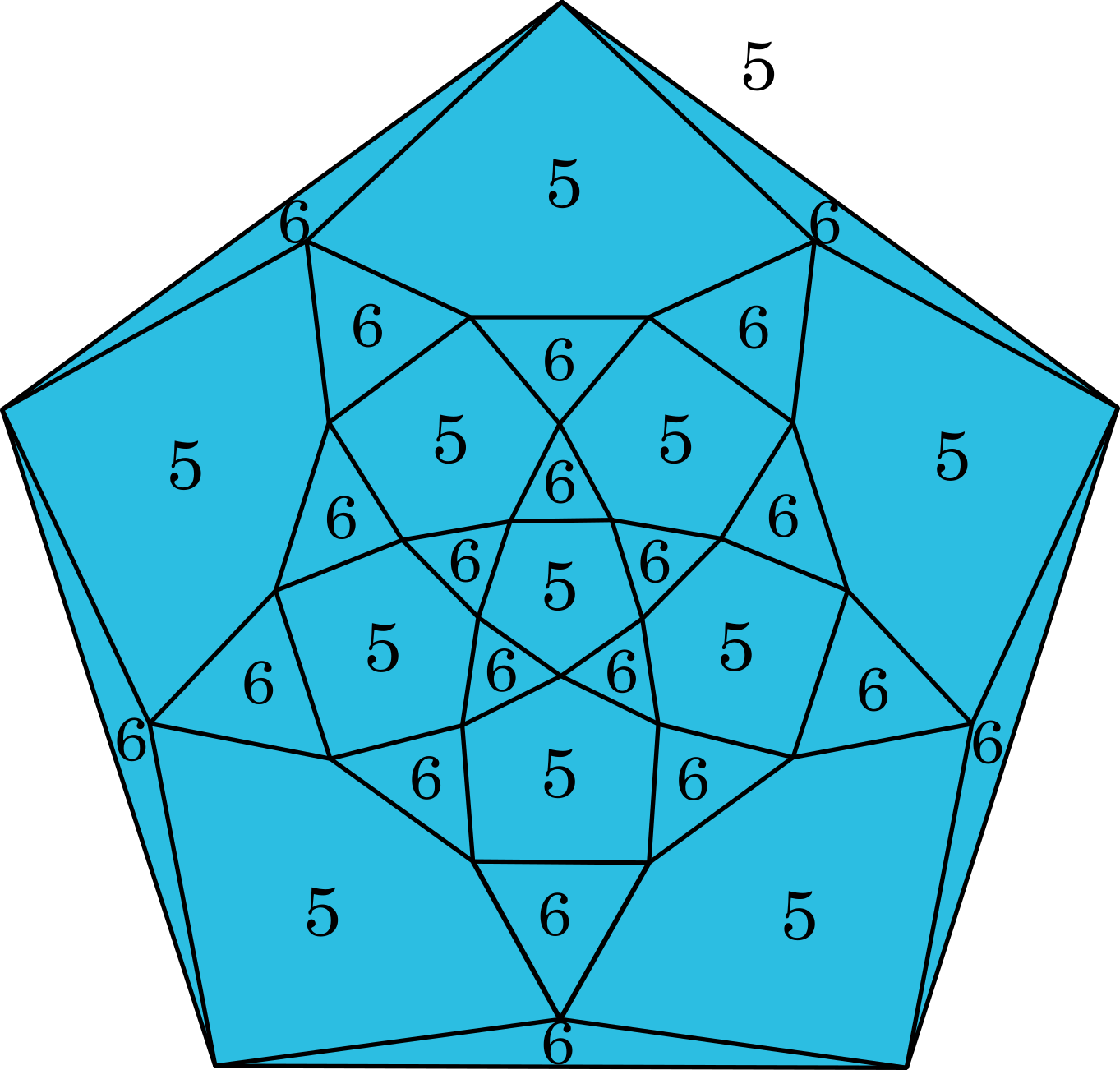}
    \caption{The Schlegel diagram of $Z(P)^\circ$, in the case where $P$ is the icosahedron from \Cref{ex:icosahedron2}. The labels represent the degrees of the polynomials of $\partial_a IP$.}
    \label{fig:schlegel_ico}
\end{figure}
We visualize the structure of these pieces using the Schlegel diagram in \Cref{fig:schlegel_ico}, where the numbers correspond to the degree of the polynomials, as explained in \Cref{rem:zonotope_combinatorics}.
\end{example}

Using this technique we are then able to visualize the boundary of intersection bodies of $4$-dimensional polytopes via the Schlegel diagram of $Z(P)^\circ$.

\begin{example}\label{ex:schlegel}
Let $P = \conv\{(1,1,0,0), (0,1,0,0), (0,-1,0,0), (0,0,-1,0), (0,0,0,-1)\}$. The boundary of its intersection body $IP$ is subdivided in $16$ regions. In four of them the equation is given by a polynomial of degree $3$, whereas in the remaining twelve regions the polynomial has degree $5$. 
In \Cref{fig:schlegel_diagrams} we show the Schlegel diagram of
\begin{equation*}
    Z(P)^\circ = \conv \{\pm (1/2, -1/2, 0, 0), \pm (1, 0, 0, 0), \pm (0, 0, 1, 0), \pm (0, 0, 0, 1)\}
\end{equation*}
with a number associated to each maximal cell which represents the degree of the polynomial in the corresponding region of $\partial IP$.
\end{example}

\begin{figure}[!ht]
    \centering
    \begin{subfigure}{0.4\textwidth}
    \centering
    \includegraphics[width=0.58\textwidth]{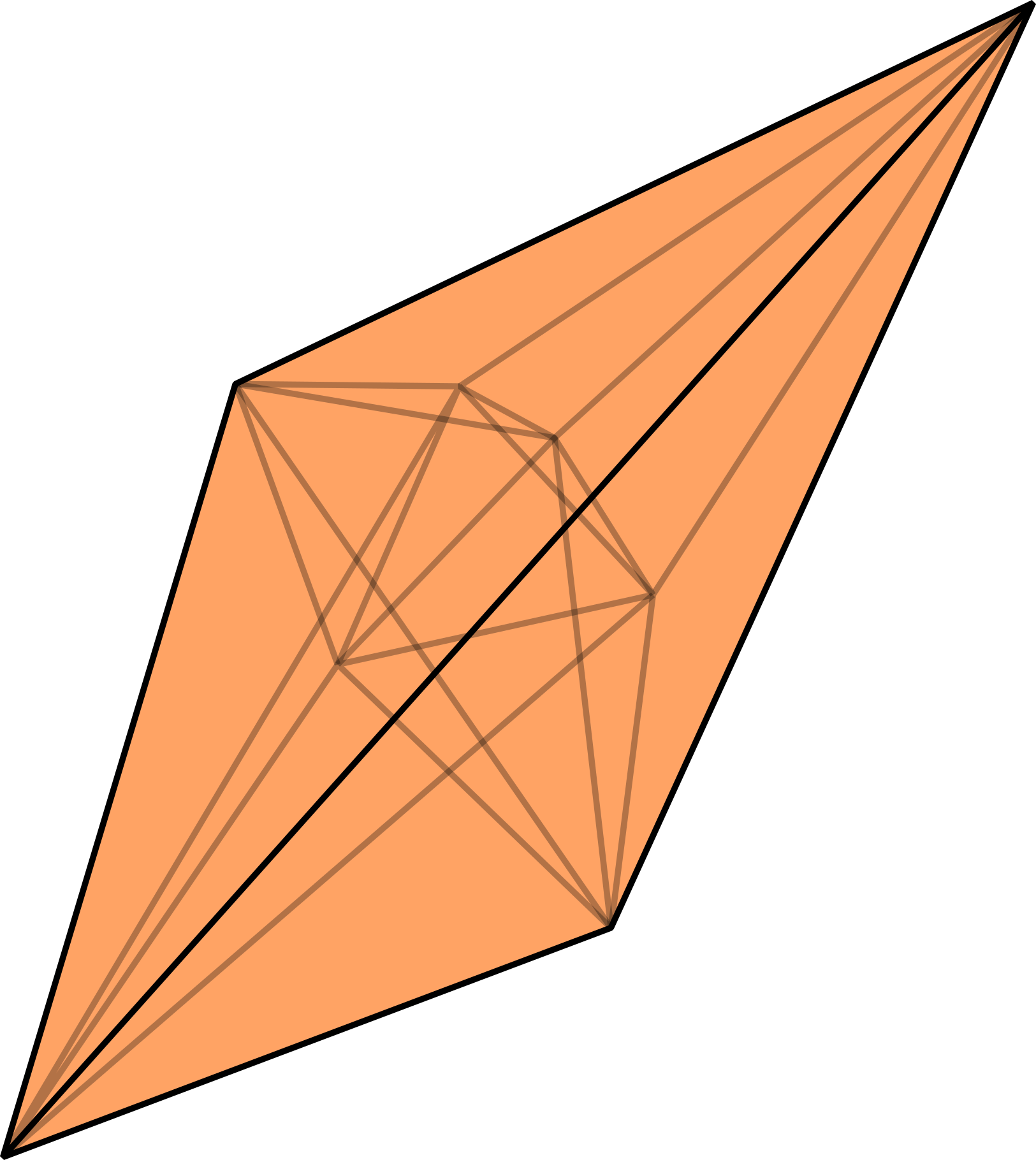}
    \caption{}
    \label{fig:schlegel}
    \end{subfigure}
    \begin{subfigure}{0.58\textwidth}
    \centering
    \includegraphics[width=0.56\textwidth]{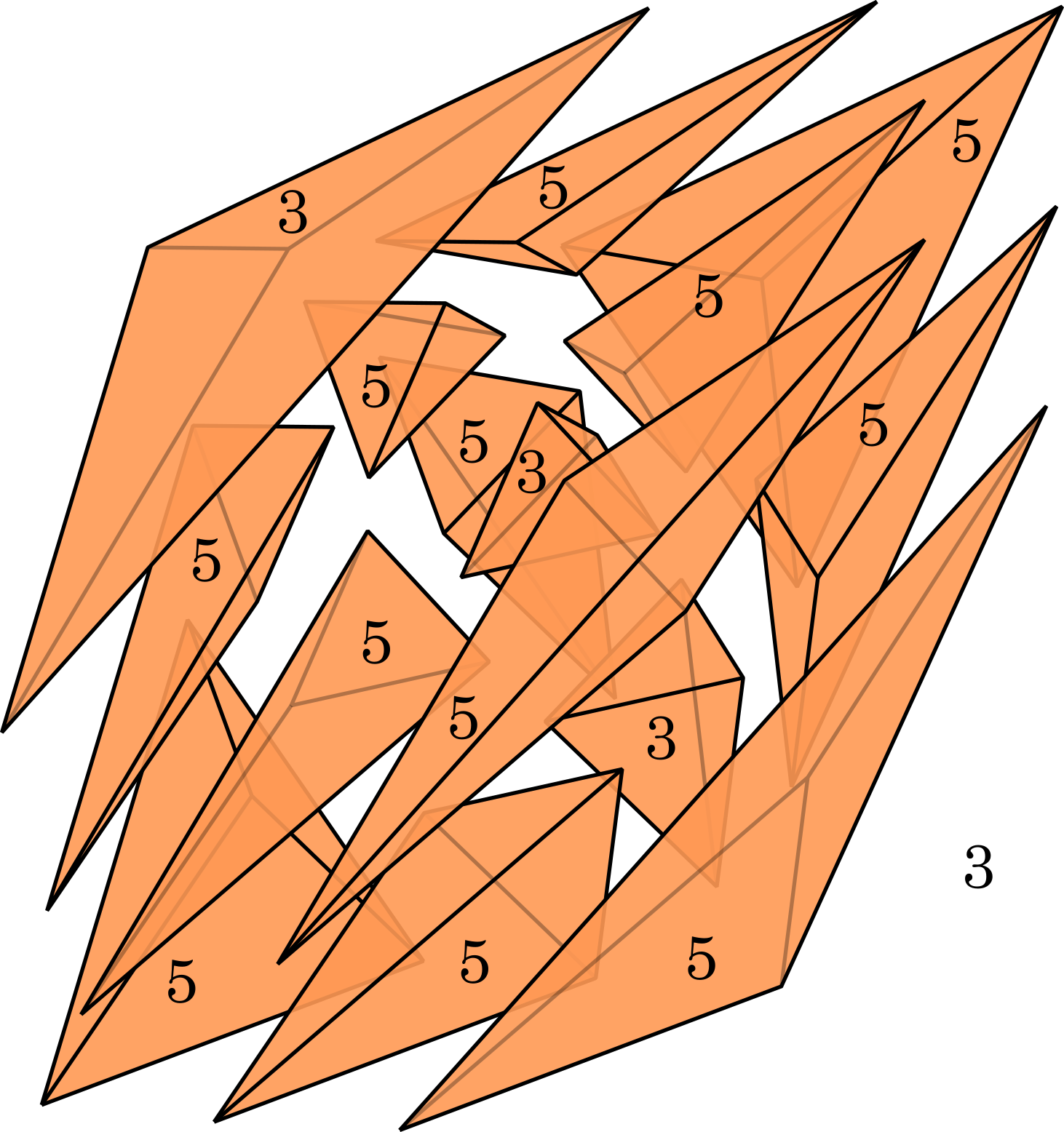}
    \caption{}
    \label{fig:schlegel_explode}
    \end{subfigure}
    \caption{The Schlegel diagram of $Z(P)^\circ$ from \Cref{ex:schlegel}. There are four cells whose corresponding polynomial in $\partial IP$ has degree $3$, including the outer facet; the others correspond to degree $5$ polynomials.}
    \label{fig:schlegel_diagrams}
\end{figure}

\section{The Cube}
\label{section:cube}
In this section we investigate the intersection body of the $d$-dimensional cube $\cube{d} = [-1,1]^d$, with a special emphasis on the linear components of its algebraic boundary.

\begin{prop}\label{prop:cube-linear-components}
The algebraic boundary of the intersection body of the $d$-dimensional cube $\cube{d}$ has at least $2d$ linear components.  These components correspond to the $2d$ open regions from \Cref{lemma:subdivision_of_sphere} which contain the standard basis vectors and their negatives.  
\end{prop}
\begin{proof}
We show the claim for the first standard basis vector $e_1$. The argument for the other vectors $\pm e_i, i = 1,\ldots , d$ is analogous.\\

Let $C$ be the region from \Cref{lemma:subdivision_of_sphere} which contains $e_1$ and consider $U = C\cap S^{d-1}$.  For any $u \in U$, the polytope $\cube{d} \cap u^\perp$ is combinatorially equivalent to $\cube{d-1}$.  Hence we can compute the (signed) volume, 
$$\vol(\cube{d} \cap u^\perp) = \det \begin{bmatrix}
v^{(1)} - v^{(0)} \\
\vdots \\
v^{(d-1)} - v^{(0)}\\
u
\end{bmatrix}$$
where $v^{(0)}$ is an arbitrarily chosen vertex of $\cube{d} \cap u^\perp$ and the remaining $v^{(i)}$ are vertices of $\cube{d} \cap u^\perp$ adjacent to $v^{(0)}$.
Next, we observe that for any vertex $v$ of $\cube{d} \cap u^\perp$ which lies on the edge $[a,b]$ of $\cube{d}$, $v$ is the vector
$$ v =  \left(-\frac{1}{u_1}\sum_{j=2}^d a_ju_j,\, a_2,\, \ldots,\, a_d \right).$$
This follows from the formulation of $v$ in the proof of \Cref{thm:intersection_body_semialgbraic} and the fact that $b_1 = -a_1$ and $b_i = a_i$ for $i=2,\ldots,d$.  Combining this with the determinant above gives us the following expression for the radial function restricted to $U$:
\begin{align*}
    \rho(u) & = \frac{1}{u_1}\det \begin{bmatrix}
-\sum_{j=2}^d (a^{(1)}_j-a^{(0)}_j)u_j & a^{(1)}_2-a^{(0)}_2 & \cdots & a^{(1)}_d-a^{(0)}_d \\
-\sum_{j=2}^d (a^{(2)}_j-a^{(0)}_j)u_j & a^{(2)}_2-a^{(0)}_2 & \cdots & a^{(2)}_d-a^{(0)}_d \\
\vdots & \vdots & & \vdots\\
-\sum_{j=2}^d (a^{(d)}_j-a^{(0)}_j)u_j & a^{(d)}_2-a^{(0)}_2 & \cdots & a^{(d)}_d-a^{(0)}_d \\
u_1^2 & u_2 & \cdots & u_d
\end{bmatrix}
\end{align*}
where we assume the determinant is nonnegative, else we will multiply by $-1$.
Expanding the determinant along the bottom row of the matrix yields
\begin{align*}
    \rho(u) &= \frac{1}{u_1}\left( u_1^2\det \begin{bmatrix}
 a^{(1)}_2-a^{(0)}_2 & \ldots & a^{(1)}_d-a^{(0)}_d \\
 a^{(2)}_2-a^{(0)}_2 & \ldots & a^{(2)}_d-a^{(0)}_d \\
& \vdots \\
 a^{(d)}_2-a^{(0)}_2 & \ldots & a^{(d)}_d-a^{(0)}_d 
\end{bmatrix} + \gamma(u_2,\dots,u_n)\right).
\end{align*}
where $\gamma(u_2,\dots,u_d)$ is a polynomial consisting of the quadratic terms in the remaining $u_i$'s.
Note that since $\gamma$ does not contain the variable $u_1$ and
$\rho$ is divisible by the quadric $u_1^2 + \ldots + u_d^2$ by \Cref{prop:degree}, it follows that 
\begin{equation}\label{eq:rho_cube_matrix}
\rho(u)  = \frac{u_1^2 + \ldots + u_d^2}{u_1} \ \det \begin{bmatrix}
 a^{(1)}_2-a^{(0)}_2 & \ldots & a^{(1)}_d-a^{(0)}_d \\
 a^{(2)}_2-a^{(0)}_2 & \ldots & a^{(2)}_d-a^{(0)}_d \\
& \vdots \\
 a^{(d)}_2-a^{(0)}_2 & \ldots & a^{(d)}_d-a^{(0)}_d 
\end{bmatrix}.
\end{equation}
Let $A$ be the $(d-1)\times (d-1)$-matrix appearing in this last expression \eqref{eq:rho_cube_matrix}.
Then finally, by the discussion in \Cref{section:algebraic_boundary}, the irreducible component of the algebraic boundary on the corresponding conical region $C$ is described by the linear equation $x_1 = |\det A|$.
\end{proof}

Note that for an arbitrary polytope $P$ of dimension at least $3$, the irreducible components of the algebraic boundary $\partial_a IP$ cannot all be linear. This is implied by the fact that the intersection body of a convex body is not a polytope. It is thus worth noting that the intersection body of the cube has remarkably many linear components. We now investigate the non-linear pieces of $\partial_a I\cube{4}$ of the $4$-dimensional cube.

\begin{example}
Let $P$ be the $4$-dimensional cube $[-1,1]^4$ and $IP$ be its intersection body. The associated hyperplane arrangement has $8 + 32 + 64 = 104$ chambers. The first $8$ are spanned by $6$ rays and the boundary here is linear, i.e. it is a $3$-dimensional cube. For example, the linear face exposed by $(1,0,0,0)$ is cut out by the hyperplane $w = 8$.

The second family of chambers is made of cones with $5$ extreme rays, where the boundary is defined by a cubic equation with shape
\begin{equation*}
    3 x y z - 3 w^2 - 6 x^2 - 12 x y - 6 y^2 - 12 x z + 12 y z - 6 z^2.
\end{equation*}
Finally there are $64$ cones spanned by $4$ rays such that the boundary of the intersection body is a quartic, such as
\begin{multline*}
    4 w x y z - w^3 - 3 w^2 x - 3 w x^2 - x^3 - 3 w^2 y - 6 w x y - 3 x^2 y - 3 w y^2 - 3 x y^2 \\
    - y^3 - 3 w^2 z - 6 w x z - 3 x^2 z + 18 w y z - 6 x y z - 3 y^2 z - 3 w z^2 - 3 x z^2 - 3 y z^2 - z^3.
\end{multline*}
\end{example}

\Cref{prop:cube-linear-components} gives a lower bound on the number of linear components of the algebraic boundary of $I\cube{d}$. We conjecture that for any $d \in \N$, the algebraic boundary of the intersection body of the $d$-dimensional cube centered at the origin has exactly $2d$ linear components. Computational results for $d \leq 5$ support this conjecture, as displayed in \Cref{tab:cube-degrees}. 
It shows the number of irreducible components of $I\cube{d}$ sorted by the degree of the component, for $d=2,3,4,5$. The first two columns are the dimension of the polytope, and the number of chambers of the respective hyperplane arrangement $H$. The third column is the degree bound from \Cref{cor:degree_cc}. The remaining columns show the number of regions whose equation in the algebraic boundary have degree $\deg$, for $\deg=1,\dots, 5$. 
\begin{table}[!ht]
\centering
\begin{tabular}{|c c c | r c c c c|}
    \hline
    dimension & \# chambers & degree bound & $\deg=1$ & $2$ & 3 & 4 & 5 \\ \hline
     2        &  4         & 1            & 4 & 0  & 0  &  0 &  0   \\
     3        & 14         & 5            & 6 & 0  & 8 &  0 &  0    \\
     4        & 104        & 14           & 8 &  0 & 32& 64&    0 \\
     5        & 1882       & 38           & 10 & 0  & 80 & 320 & 1472    \\
     \hline
\end{tabular}
\caption{Number of irreducible components of the algebraic boundary of the intersection body of the $d$-cube, listed by degree.}
\label{tab:cube-degrees}
\end{table}

It is worth noting that the highest degree attained in these examples is equal to the dimension of the respective cube. In particular, the degree bound for centrally symmetric polytopes, as given in \Cref{cor:degree_cc} is not attained in any of the cases for $d\geq 3$.
Finally, note that the number of regions grows exponentially in $d$, and thus for $d\geq 3$, the number of non-linear components exceeds the number of linear components.

\vspace{1cm}
\section*{Acknowledgments}

The authors would like to thank Rainer Sinn, Bernd Sturmfels and Simon Telen for many useful discussions and support. We are grateful to Michael Joswig and Lars Kastner for their time and their help with OSCAR. We also wish to thank the referee for their insight and feedback.  Last, thank you to the Max Planck Institute for Mathematics in the Sciences (MPI MiS) where the research for this project was done.

\newpage
\bibliographystyle{alpha}
\bibliography{references}

\vspace*{\fill}
\subsection*{Affiliations}
\vspace{0.2cm}
\noindent \textsc{Katalin Berlow} \\
\textsc{University of California, Berkeley} \\
\url{katalin@berkeley.edu} \\

\noindent \textsc{Marie-Charlotte Brandenburg} \\
\textsc{ Max Planck Institute for Mathematics in the Sciences \\
Inselstra{\ss}e 22, 04103 Leipzig, Germany} \\
 \url{marie.brandenburg@mis.mpg.de} \\

\noindent \textsc{Chiara Meroni} \\
\textsc{ Max Planck Institute for Mathematics in the Sciences \\
Inselstra{\ss}e 22, 04103 Leipzig, Germany} \\
\url{chiara.meroni@mis.mpg.de} \\

 \noindent \textsc{Isabelle Shankar} \\
\textsc{ Max Planck Institute for Mathematics in the Sciences \\
Inselstra{\ss}e 22, 04103 Leipzig, Germany} \\
\url{isabelle.shankar@mis.mpg.de}

\end{document}